\documentclass[]{gSSR2e}
\newcommand*{\Ab}{{\mathbb A}}

\newcommand*{\Rb}{{\mathbb R}}
\newcommand*{\la}{{\langle}}
\newcommand*{\ra}{{\rangle}}

\newcommand*{\R}{{\mathbb R}}

\begin{document}
\doi{10.1080/1744250YYxxxxxxxx}
 \issn{1744-2516}
\issnp{1744-2508}
\jvol{00} \jnum{00} \jyear{YYYY} \jmonth{Month}
\markboth{Curry, Ebrahimi--Fard, 
Malham  and Wiese}{L{\'e}vy processes and quasi-shuffle algebras}


\title{\itshape L{\'e}vy Processes and Quasi-Shuffle Algebras}
\author{Charles Curry$^{\rm a}$, Kurusch Ebrahimi--Fard$^{\rm b}$, 
Simon J.A.~Malham$^{\rm a}$  and 
Anke Wiese$^{\rm a}$$^{\ast}$\thanks{$^\ast$Corresponding author. 
Email: A.Wiese@hw.ac.uk}\thanks{Acknowledgements. 
KEF, SJAM and AW  would like to thank the 
Edinburgh Mathematical
Society for support for a visit by KEF to Heriot--Watt in 
July 2012.}\vspace{6pt}
$^{\rm a}${\em Maxwell Institute for Mathematical Sciences
and School of Mathematical and Computer Sciences,
Heriot-Watt University, Edinburgh EH14 4AS, UK}; 
$^{\rm b}${\em Instituto de Ciencias Matem\'aticas, 
Consejo Superior de Investigaciones Cient\'{i}ficas, 
C/ Nicol\'as Cabrera, no. 13-15, 28049 Madrid, Spain}
 \\\vspace{6pt}\received{\rm{(}7 October 2013\rm{)}\\
\vspace{6pt}{\bf Accepted for publication in {\em Stochastics: An International Journal of Probability and Stochastic Processes}}}}

\maketitle

\begin{abstract}
We investigate the algebra of repeated integrals of semimartingales. 
We prove that
a {\em minimal family} of semimartingales generates a 
 quasi-shuffle algebra. 
In essence, to fulfill the minimality criterion, first, the family must be 
a minimal generator of the algebra of repeated integrals 
generated by its elements and by
quadratic covariation processes recursively constructed from the
elements of the family.
Second,
recursively constructed quadratic covariation processes may lie in 
the linear
span  of 
previously constructed ones and of the family,
but may  not lie in the linear span of repeated integrals of these.
We prove that a finite family of independent
L\'evy processes that have finite moments 
generates a minimal family. Key to the
proof are the
Teugels martingales and a strong
orthogonalization of them. We conclude that a finite family
of independent L\'evy processes form a quasi-shuffle algebra.
We discuss important potential applications to constructing
efficient numerical methods for the strong approximation
of stochastic differential equations driven by L\'evy
processes.
\end{abstract}

\begin{keywords}{semimartingales, L\'evy processes, quasi-shuffle algebra,
Teugels martingales}
\end{keywords}
\begin{classcode} 60H30, 60G44
\end{classcode}

\section{Introduction}\label{intro}

The set of all multiple Stratonovich integrals constructed from 
independent Wiener processes
generates a shuffle algebra (see Gaines 1994). 
This is because the usual integration by
parts formula holds for such multiple integrals. The set of all multiple
It\^o integrals on the other hand generates a quasi-shuffle algebra. 
In this case the non-zero quadratic variation of the underlying 
Wiener processes is revealed by the It\^o 
integration by parts formula. As might
be expected, the two algebras are isomorphic (by direct application
of the results
by Hoffman 2000).  More generally,
 Li \& Liu 1997 studied algebraic bases for
independent Wiener and Poisson processes  and the set of multiple 
It\^o integrals constructed from them. 

Knowledge about the algebraic structure of stochastic systems has 
been proved to be very useful 
in a range of applications. Recent applications 
 include chaotic
representations of martingales (Jamshidian 2011), the generalization
 and study of the concept of a Fliess operator to 
input-output maps driven  by It{\^o} processes (Duffaut Espinosa, Gray 
\& Gonz{\'a}lez 2012), and the design and analysis of
efficient stochastic simulation methods for stochastic differential
equations driven by Wiener processes (Malham \& Wiese 2009 and 
Ebrahimi-Fard et al.~2012), among others.

It is natural to now ask the 
question of whether a family of Wiener--Poisson processes or 
more generally of L\'evy processes
generates a quasi-shuffle algebra.  Indeed what about a family
of semimartingales? In this paper we prove 
the following new main results, that a collection of:
\begin{enumerate}
\item[(1)] Semimartingales that generate
a \emph{minimal family} (see Section~\ref{minimal_family})
form a quasi-shuffle algebra;
\item[(2)] L\'evy processes that have finite moments 
generate a minimal family.
\end{enumerate}
A natural consequence is that the \emph{Hoffman exponential map} 
(see Hoffman 2000 and Hoffman \& Ihara 2012) establishes 
an isomorphism between the quasi-shuffle algebra of 
L\'evy processes and a shuffle algebra.
From a practical strong simulation perspective knowledge of the 
quasi-shuffle structure  is highly desirable.
This is because in principle, we can utilize the convolution shuffle algebra
analysis of Ebrahimi-Fard \emph{et al.~}(2012) to establish efficient
strong integrators for stochastic differential equations driven by 
L\'evy processes. L\'evy processes have become increasingly popular
in recent years and can now be regarded as 
one of the key ingredients for many models
in finance and economics and in insurance; their efficient 
simulation 
has thus become an important aspect of these applications.

Let us outline the key points more explicitly. 
Suppose we are given a finite family of semimartingales.
Without loss of generality we assume they are all zero at time $t=0$.
The real product of two semimartingales $X$ and $Y$ 
is given by
$XY=\int X_-\,\mathrm{d} Y+\int Y_-\,\mathrm{d} X+[X,Y]$.
Here $[X,Y]$ is the quadratic covariation of $X$ and $Y$
and represents the It\^o correction to the classical integration 
by parts formula; it is itself a semimartingale, 
and the space
of semimartingales with multiplication forms an algebra.
The formula above 
is reminiscent of a quasi-shuffle 
of $X$ and $Y$, see Section~\ref{quasi_shuffle}, 
and see Hoffman 2000, Ebrahimi--Fard \& Guo 2006, 
Novelli, Patras \& Thibon 2011, and Hoffman \& Ihara 2012 
for more details on this product.
Indeed, for a {\em minimal} family of semimartingales, we can assign a 
letter to each
element in the family, and inductively, new letters to those nested quadratic 
covariation processes that are new and not linear combinations of those 
hitherto 
constructed.
In so doing, we define an alphabet
$\Ab$, and we then establish an isomorphism 
from the quasi-shuffle algebra $\Rb\la\Ab\ra$ 
of noncommutative polynomials and
formal power series  generated by 
 words formed from $\Ab$
to the algebra of repeated integrals
generated  by the given semimartingales.

Our second main result is that a family of independent
\emph{L\'evy processes} with finite moments 
generates a minimal family. Key
to establishing
this result are Teugels martingales and a 
strong orthogonalization of them. 
Here we relied particularly on work by Nualart \& Schoutens (2000),
Davis (2005), 
 and Jamshidian (2005). We 
conclude a family of independent L\'evy processes generates
a quasi-shuffle algebra.

Our paper is structured as follows. 
We introduce our notion of minimal families
of semimartingales in Section~2.
In Section~3 we prove that a family of independent L\'evy processes 
generate a minimal family. We characterize those L\'evy processes 
for which the alphabet $\Ab$ is finite.  
With the concrete general example of independent L\'evy processes in hand,
we then establish in Section~4 the isomorphism from the quasi-shuffle
algebra to the algebra generated by
 a minimal family  of semimartingales. We then apply this result
to a family of independent L{\'e}vy processes and derive further algebraic
properties of the algebra generated by them.
Finally in Section~5 we conclude and 
discuss important applications of our results.

\section{Minimal families of semimartingales}\label{minimal_family}

Underlying our analysis is a complete filtered probability space 
$\big(\Omega, {\mathcal F}, \big({\mathcal F}_t\big)_{t\ge 0}, P)$
 satisfying what is known as the
{\em usual conditions} of
completeness and right-continuity, see Protter (1992; p.~3). 
Without loss of generality, we
assume that ${\mathcal F}_0$ is generated by the $P$-null sets.
Due to the usual conditions  every martingale has
a modification that has paths that are right-continuous with left limits 
(see Protter p.~5, Corollary 1). 
We assume  henceforth all martingales have this property.

A process $X$ is a {\em semimartingale}, if $X$ has a decomposition 
$X_t=X_0+M_t+A_t$
for $t\ge 0$,
where $M_0=A_0=0$,  
and where $M$ is a local martingale  and $A$ is an adapted process that is 
right-continuous
with left limits and 
has finite variation on each finite interval $[0, t]$.  
Recall that a process $A$ is {\em predictable}, if it is 
measurable with respect
to the $\sigma$-algebra on ${\mathbb R^+}\times \Omega$ generated by the
left-continuous processes. 
A semimartingale $X$ that admits a decomposition 
with a predictable 
finite variation process $A$ 
 is  a {\em special semimartingale}. Such a 
decomposition is unique, see Jacod \& Shiryaev (2002, Definition I.3.1  and
I.4.21).

The space of semimartingales 
with multiplication forms an algebra (Protter p.~60, 
Corollary~3). 
The  
 {\em quadratic covariation} or {\em square bracket process}
$[X, Y]$ 
between two semimartingales $X$ and $Y$ is defined via their product
as follows
\begin{eqnarray}\label{eq:covariation}
XY=X_0Y_0+\int X_-\, \mathrm{d}Y+\int Y_-\, \mathrm{d}X + [X,Y],
\end{eqnarray}
see Protter (1992; p.~58). The quadratic covariation 
of a process $X$ with itself is known as its {\em quadratic
variation}
(we refer to the
monographs by Protter (1992) and Jacod \& Shiryaev (2002) for details).
The following property, which is essential for the definition
of the quasi-shuffle product, follows
from 
I.4.49 Proposition and I.4.52 Theorem in Jacod \& Shiryaev (2002). 

\begin{remark} 
\label{variation_properties}
Let $X$, $Y$ and $Z$ be semimartingales, then:
 (a)~$[X, 0]  = 0$; (b)~$[X, Y]  = [Y, X]$;
(c) $[X, [Y, Z]]   
= \sum_s \Delta X_s \Delta Y_s \Delta Z_s  =  [[X, Y], Z]$.
Here $\Delta X_s:= X_s-\lim_{u\nearrow s} X_u$
denotes the jump of a semimartingale $X$ at time $s$.

Hence, the quadratic covariation defines a {\em commutative, 
associative product}  on the space of semimartingales.
\end{remark}

Let ${\mathcal X}=\{X^1, \, \ldots , \, X^d\}$ be a finite family of 
semimartingales. By considering $X^i-X^i_0$, we
may and will assume $X^i_0=0$ for all $i=1,\,\ldots,\, d$. Let 
${\mathcal A}$ denote the algebra of repeated integrals 
generated by ${\mathcal X}$.
Important elements of the algebra ${\mathcal A}$ are the multiple 
bracket processes (see Jamshidian 2005). 

\begin{definition}{\bf (Power Bracket)}
For a semimartingale $X$,  we denote $[X]^{(1)}=X$, and
 we
define the {\em power bracket}  for $n\ge 2$ by
$[X]^{(n)} = [X, [X]^{(n-1)}]$.
\end{definition}

Note for $n\ge 3$, the $n$-bracket is given by 
$[X]^{(n)} =\sum \big(\Delta X\bigr)^n$. This
is also known as the  {\em power jump process} (see Nualart \& Schoutens 2000).
Key to relating the semimartingale product to a quasi-shuffle product
is the following property.

\begin{definition}{\bf (Minimal family)} 
A family ${\mathcal X}$ of semimartingales
is {\em minimal} if it satisfies the following two conditions: 
\begin{enumerate}
\item[(A)] {\bf (Minimal Generator)} It is the minimal 
generator of the algebra ${\mathcal A}$ of repeated integrals generated by 
${\mathcal X}$ and by 
successively constructed nested covariation processes from ${\mathcal X}$.
\item[(B)] {\bf (Consistency)} 
Successively
constructed quadratic covariation processes may lie in 
the linear
span  of 
previously constructed ones and of the family,
but may  not lie in the linear span of repeated integrals of these.
\end{enumerate} 
\end{definition}

\begin{remark}\label{interpretation_minimal}
Property (B) is essential for the definition
 of the alphabet underlying the
quasi-shuffle algebra in Section~\ref{quasi_shuffle}. In order that
the multiplication of semimartingales defines a quasi-shuffle product, 
the nested covariation processes 
have to be assigned 
new letters in the alphabet except if they are in
the linear span of other
letters in the alphabet (that is if they are linear combinations of
previously constructed quadratic covariation processes and the 
semimartingales themselves).  Property (B) excludes families of
semimartingales for which 
quadratic covariation processes would  correspond to words, that is
multiple integrals,  
rather than letters in the quasi-shuffle algebra.  
\end{remark}

\begin{example}\label{example_minimal}  
We give several examples to illustrate the concept of
minimality.
\begin{enumerate}
\item[(a)] Suppose the family of semimartingales ${\mathcal X}$ is generated 
by independent Wiener processes $W^1,\, \ldots ,\, W^d$.  
Then $[W^i, W^j]_t =\delta_{ij}t $ and $[W^i]^{(n)}=0$ for
all $n\ge 3$. Thus ${\mathcal X}$ is minimal.
\item[(b)] Suppose the family of semimartingales ${\mathcal X}$ is generated 
by independent Poisson processes $P^1,\, \ldots ,\, P^d$. Then 
$[P^i, P^j] =\delta_{ij} P^i$, and $[P^i]^{(n)}=P^i$ for
all $n\ge 2$. Thus ${\mathcal X}$ is minimal.
\item[(c)] More generally, we will show in Section~\ref{levy_processes} 
 that independent L{\'e}vy processes with moments of all orders
generate a minimal family of semimartingales. 
\item[(d)] Let
$X^1_t =  W_t$ and
$X^2_t = \int_0^t W_s\,\mathrm{d}s$,
where $W$ is a Wiener process.
Then $[X^1, X^1]_t = t$, and $X^2_t=\int_0^t X^1_s\mathrm{d}[X^1,X^1]_s$.  
Hence $\{X^1, X^2\}$ is {\em not} minimal. 
However, the family generated by $\{X^1\}$ only is minimal.
\end{enumerate}
\end{example}

\section{Minimal families of L{\'e}vy processes}\label{levy_processes}
We consider the  case of $d$ 
independent L{\'e}vy processes $X^1, \ldots ,\, X^d$. 
We will assume here and in the sequel that all $X^i$ have 
moments of all
orders.  Without
loss of generality, assume that all $X^i$ are stochastic. 
The goal of this section is to show that 
${\mathcal X}=\{X^1, \ldots ,\, X^d\}$ is a minimal family. 
If one of the processes 
$X^i$ is continuous, the process $t$ will be generated by 
$[X^i]^{(2)}$. Otherwise, if none
of the $X^i$ is continuous, then we augment $\{X^1, \ldots ,\, X^d\}$ with $t$,
and
$\{t,\, X^1, \ldots ,\, X^d\}$ will be minimal.

Recall that a  L{\'e}vy process is zero
at time 0, has independent stationary increments and is continuous 
in probability. 
Since the processes $X^i$ are independent, condition (A) in the definition
of a minimal family is satisfied. To show that condition (B) is fulfilled, 
we assume first that $d=1$, and write $X=X^1$. 
The L{\'e}vy process $X$ can be characterised by  the 
triplet $(\alpha,\, \sigma^2,\, \nu)$, where $\alpha $ 
and $\sigma$ are constants and where $\nu$ is  a measure on $\mathbb R$ 
with $\nu(0)=0$ satisfying 
$\int_{\mathbb R} \inf(1, x^2)\, \nu(\mathrm{d}x)<\infty$. 
The L{\'e}vy Decomposition Theorem (see e.g.~ Protter 1992 Theorem I.42)
states that the process 
$X$ has a unique decomposition
\begin{equation*}
X_t=\alpha t +\sigma W_t + J_t,
\end{equation*}
where $W$ is a Wiener process and $J$ is a purely discontinuous martingale.
Note in particular that L{\'e}vy processes are special semimartingales. 
In terms of the L{\'e}vy measure
$\nu$, we can express $J$ in the form
\begin{equation*}
J_t = \int_0^t \int_{\mathbb R} x\,
\big(Q(\mathrm{d}s, \mathrm{d}x)-\,\mathrm{d}t\,\nu(\mathrm{d}x)\big),
\end{equation*}
where $Q(\mathrm{d}t, \mathrm{d}x)$ is a random Poisson measure with intensity measure 
$\mathrm{d}t\times\nu(\mathrm{d}x)$. 
Theorem 1.4.52 in Jacod \& Shiryaev (2002) and 
Theorem I.36 in Protter (1992)   
imply that
the quadratic variation and power brackets of $X$ are given by
\begin{align*}
[X]^{(n)}_t & = \, \sigma^2 1_{\{n=2\}} t+ \sum_{0\le s \le t}\big(\Delta J_s\big)^n
 = \, \sigma^2 1_{\{n=2\}} t
+\int_0^t \int_{\mathbb R} x^n \,Q(\mathrm{d}s, \mathrm{d}x)
\end{align*}
for $n\ge 2$. 
Note that the power bracket is again a L{\'e}vy process,  
and that for $n\ge 2$
$E\bigl[[J]^{(n)}_t\bigr] = t \int_{\mathbb R} x^n\, \nu(\mathrm{d}x)< \infty$,
see Nualart \& Schoutens (2000, p.~111). 
For $n\ge 2$, set 
$\alpha_n = \int_{\mathbb R} x^n\, \nu(\mathrm{d}x)$,
and denote the compensated power jump processes by
\begin{align*}
Y^{(1)}_t & = \, [X]^{(1)}_t-\alpha t\qquad\text{and}\qquad
Y^{(n)}_t  = \, [J]^{(n)}_t-\alpha_n t.
\end{align*}
The processes $Y^{(n)}$, $n\ge 1$, are known as the 
{\em Teugels martingales}, see Nualart \& Schoutens (2000). 
The power bracket  has thus the unique decomposition
$[X]^{(n)}_t =\, (\sigma^21_{\{n=2\}} + \alpha_n )t + Y^{(n)}_t$ 
as the sum of the purely discontinuous martingale $Y^{(n)}$ and the 
deterministic process $(\sigma^21_{\{n=2\}} + \alpha_n) t$. 

Two locally  square-integrable martingales are 
{\em strongly orthogonal}, if their (real) product is a local martingale 
(see Jacod \& Shiryaev 2002, I.4.11 Definition). 
Since $X$ is assumed to have moments of all orders, 
all compensated power processes $Y^{(n)}$ are square-integrable 
martingales. Hence one can find pairwise strongly orthogonal square-integrable 
martingales
$H^{(i)}, i\ge 1$, 
such that 
\begin{equation}\label{eq:strongly_orthogonal}
Y^{(n)} = c_{n,1}H^{(1)} + c_{n,2}H^{(2)} + \ldots + H^{(n)}
\end{equation}
for $n\ge 1$, where $c_{n,i}$ are constants with $c_{n,n}=1$, see 
Nualart \& Schoutens (2000). 

\begin{remark}\label{rm:orthogonality}
A standard procedure to construct the 
strongly orthogonal martingales $H^{(n)}$ 
is as follows (see Davis \& Varaiya 1974, Davis 2005,
Nualart \& Schoutens 2000  and Jamshidian 2005). 
For a locally square-integrable martingale $M$, there exists a unique 
predictable increasing process $\la M, M\ra$, the 
{\em sharp or angular bracket} of $M$, 
such that $M^2-\la M, M\ra $ 
is a local martingale. 
By direct calculation,
the sharp bracket of the Teugels 
martingales $Y^{(i)}$ and $Y^{(j)}$ is given by 
$\la Y^{(i)},\, Y^{(j)}\ra_t =(\alpha_{i+j} +\sigma^21_{\{i=j=1\}})\cdot t$,
see equation (1.7) in Davis (2005). Define inductively 
$H^{(1)}:= \, Y^{(1)}$ and
$H^{(n)}  := Y^{(n)} - \sum_{k=1}^{n-1} \int 
\frac{\mathrm{d}\la Y^{(n)},\, H^{(k)}\ra}{\mathrm{d}\la H^{(k)},\, H^{(k)}\ra}
\, \mathrm{d}H^{(k)}$
for $n \ge 2$. 
Importantly,
the sharp brackets $\la H^{(k)},\, H^{(k)}\ra_t$ are scalar multiples of $t$. 
It follows
 for a square-integrable process $\varphi$ we have
$\| \int_0^t\varphi_s\, \mathrm{d}H^{(n)}_s \|_{L^2(P)} = 0$ for all $t\ge 0$
 if and only if $\varphi\equiv 0$. 
A standard localisation procedure now implies 
the following  identity, which is essential for the 
proof of Theorem~\ref{th:minimality}. 
Let $n\ge 1$, and assume that $H^{(1)}, \ldots ,\, H^{(n)} \not\equiv 0$. If
$\varphi^i$, $i=1, \ldots , n$, are left-continuous processes with
$\sum_{i=1}^n\int\varphi^i_s\, \mathrm{d}H^{(i)}_s \equiv 0$, 
then $\varphi^i\equiv 0$ 
for all $i=1, \ldots , n$. 
We remark that in essence the equality 
$\sum_{i=1}^n\int\varphi^i_s\, \mathrm{d}H^{(i)}_s \equiv 0$
describes
 the Galtchouk-Kunita-Watanabe decomposition
of the martingal that is identically zero  into the sum
of $n$ orthogonal locally square-integrable stochastic integrals
$\int\varphi^i_s\, \mathrm{d}H^{(i)}_s$. This decomposition
is known to be
unique, 
see the remark on page 127 following Theor{\`e}me 4.27 in Jacod (1979).
\end{remark}

\begin{lemma}\label{linear_span}
Let $k\ge 1$. Consider the following properties:
\begin{enumerate}
\item[(a)] $H^{(k)}\equiv 0$.
\item[(b)] $Y^{(k)}$ is in the linear span of $\{Y^{(1)}, \ldots ,\, Y^{(k-1)}\}$.
\item[(c)] $[X]^{(k)}$ is in the linear span of 
$\{t, [X]^{(1)}, \ldots ,\, [X]^{(k-1)}\}$.  
\item[(d)] $[X]^{(n)}$ is in the linear span of
$\{t, [X]^{(1)}, \ldots ,\, [X]^{(k-1)}\}$
for all $n\ge k$.
\item[(e)] $Y^{(n)}$ is in the linear span of $\{Y^{(1)}, \ldots ,\, Y^{(k-1)}\}$ 
for all $n\ge k$.
\item[(f)] $H^{(n)}\equiv 0$ for all $n\ge k$.
\end{enumerate}
Then (a) implies  properties (b) to (f). Furthermore
(a), (b) and (c) are equivalent, and (d), (e)
 and (f) are equivalent. 
\end{lemma}

\begin{proof}
The equivalence of (a), (b) and (c) and the equivalence of 
(d), (e) and (f)  follow directly from 
the definitions and relations between the power bracket processes, the Teugels
martingales and the orthogonal basis. 
We will show that (c) implies (d). 
Suppose that $[X]^{(k)}$ is in the linear span of 
$\{t, [X]^{(1)}, \ldots ,\, [X]^{(k-1)}\}$.
By definition we have
$[X]^{(k+1)}=[X, [X]^{(k)}]$, and hence  $[X]^{(k+1)}$ 
is in the linear span of
$\{t, [X]^{(1)}, \ldots ,\, [X]^{(k)}\}$,
 which by assumption on $[X]^{(k)}$ coincides with the linear span of 
$\{t, [X]^{(1)}, \ldots ,\, [X]^{(k-1)}\}$.
Hence,  we conclude inductively that 
$[X]^{(n)}$ is in the linear span of
$\{t, [X]^{(1)}, \ldots ,\, [X]^{(k-1)}\}$
for all $n\ge k$. 
\end{proof}

\begin{remark}\label{linearspan_general}
The implication from  (c) to (d) 
holds for general semimartingales. 
\end{remark}

Property (B) for the minimality of the family $\{X\}$ 
follows from the 
following stronger key property of the power brackets of $X$.

\begin{theorem} \label{th:minimality} 
Let $n\ge 1$, and
assume  there are left-continuous processes 
$\varphi^k$, $k=0,  \ldots , n-1$, 
such that
\begin{eqnarray} \label{eq:representation}
[X]^{(n)} = 
\sum_{k=1}^{n-1} \int \varphi^k_s\, \mathrm{d}[X]^{(k)}_s 
+\int \varphi^0_s\, \mathrm{d}s.
\end{eqnarray}
Then $\varphi^k$ is constant for all $k=0, \ldots ,\, n-1$. 
\end{theorem}

\begin{proof}
We can assume without loss of generality 
that $H^{(1)},\,\ldots , \, H^{(n-1)}\neq 0$. If $H^{(k)}\equiv 0$  for
some $k\ge 1$, then 
$[X]^{(j)}$ is in the linear
span of $\{t, [X]^{(1)}, \ldots , [X]^{(k-1)}\}$ for all $j\ge k$ 
by Lemma~\ref{linear_span},  and   
equation~\eqref{eq:representation} is equivalent to
$[X]^{(n)} = 
\sum_{k=1}^{k_0} \int \psi^k_s\, \mathrm{d}[X]^{(k)}_s 
+\int \psi^0_s\, \mathrm{d}s$,
where  
$\psi^k$  are suitably defined left-continuous processes, 
and where $k_0=\min\{k\ge 1: H^{(k)}\equiv 0\}$.
By compensating the power brackets, equation \eqref{eq:representation}
 is equivalent to 
\begin{equation*}
Y^{(n)}_t +\alpha_n t =   
\sum_{k=1}^{n-1} \int_0^t \varphi^k_s\, \mathrm{d}Y^{(k)}_s + 
\sum_{k=2}^{n-1} \int_0^t \varphi^k_s\big(\alpha_k+\sigma^2 1_{\{k=2\}}\big)\,
\mathrm{d}s
+\int_0^t \big(\varphi^1_s \alpha+\varphi^0_s\big)\, \mathrm{d}s.
\end{equation*}
The uniqueness of the decomposition of the 
stochastic integral into predictable finite variation process and
local martingale
yields that for all $t\ge 0$  
\begin{eqnarray}\label{eq:rep_dt}
\alpha_n   & = & \, \varphi^1_t\alpha + \varphi^0_t +
\sum_{k=2}^{n-1}  \varphi^k_t\big(\alpha_k +\sigma^2 1_{\{k=2\}}\big)
\end{eqnarray}
and
\begin{align}
Y^{(n)}_t & =  \,
\sum_{k=1}^{n-1} \int_0^t \varphi^k_s\, \mathrm{d}Y^{(k)}_s. 
\end{align}
The latter equation is equivalent to
$\sum_{i=1}^{n} c_{ni}\, H^{(i)}
 = \sum_{k=1}^{n-1} \sum_{i=1}^{k}\int \varphi^k_s c_{ki}\, \mathrm{d}H^{(i)}_s$,
with $c_{nn}=1$, and after rearrangement
$H^{(n)}-\sum_{i=1}^{n-1} \int
\sum_{k=i}^{n-1} \varphi^k_s c_{ki}-c_{ni}\, \mathrm{d}H^{(i)}_s
= 0$.
Since  $H^{(i)}\not\equiv 0$, 
we have by 
Remark~\ref{rm:orthogonality} that
$\sum_{k=i}^{n-1}\varphi^k c_{ki}-c_{ni} \equiv 0$  for all $i=1,\,
\ldots , \, n-1.$ 
Recursively, it follows that 
$\varphi^k$ is  constant for all $k=n-1,\,\ldots , \, 1$.
Equation \eqref{eq:rep_dt} now implies that $\varphi^0$ is constant.
Hence the assertion follows. 
\end{proof}

\begin{corollary}
The family $\{X\}$ is minimal. 
\end{corollary}

\begin{remark}
For $n\ge 2$ the power bracket $[X]^{(n)}$ has finite 
variation. Hence,
any integral representation of $[X]^{(n)}$ cannot include an integral 
with respect to 
the Wiener process $W$. Hence if $\sigma \not=0$, then the 
integrand $\varphi^1$ in  
\eqref{eq:representation} must be zero.
\end{remark}

Suppose for the moment that  $d=2$, and let $X$ and $Z$ be 
independent L{\'e}vy processes 
 with L{\'e}vy decompositions
$X_t  =  \alpha t +\sigma W_t +J_t$ and
$Z_t =  at + c B_t + L_t$, 
where $B$ and $W$ are independent Wiener processes and $J$ and $L$ are 
independent purely discontinuous martingales. 
In particular, the power brackets $[X]^{(n)}$ and $[Z]^{(m)}$ are independent 
for all $n, \, m$.
Hence, any representation
of the power bracket $[X]^{(n)}$ as a sum
of stochastic  integrals cannot have non-zero contributions 
from the process $Z$ and its  power 
brackets as long as these are stochastic, it can only have contributions  
through the deterministic process $t$, if any. 
Hence ${\mathcal X}=\{X, Z\}$ is a minimal family.
This property extends straightforwardly to $d$ independent 
L{\'e}vy processes. In conclusion, we have established the following.

\begin{theorem}\label{th:main_Levy}
Let $X^1, \ldots ,\,X^d$ be independent non-deterministic 
L{\'e}vy processes  with moments
of all orders. 
Then $\{X^1,\, \ldots , \, X^d\}$ is a minimal family
of semimartingales. If none of  the $X^i$, $i=1,\, \ldots,\, d$, 
is continuous, 
then $\{t, X^1, \ldots , \, X^d\}$ is also a minimal family. 
\end{theorem}

The following lemma completely
characterises L{\'e}vy processes with
  power brackets that are linear combinations of lower order power 
brackets as 
compound Poisson processes which assume
a finite number of  values only 
 (with or without  continuous component $\alpha t+\sigma W_t$). 
As an important consequence we have that  
the alphabet underlying the quasi-shuffle
algebra in Section~\ref{quasi_shuffle}
is finite  if and only if the purely discontinuous
martingale part of each $X^i$, $i=1,\ldots, \, d$, is a 
linear combination of (standard) compensated Poisson processes.

\begin{proposition} \label{compound_Poisson}
Let  $X_t=\alpha t +\sigma W_t +J_t$.
 Let $n\ge 1$. 
If the $n$-power bracket $[X]^{(n)}$  is
in the linear span of $\big\{t,\, [X]^{(1)},\,\ldots ,\,  [X]^{(n-1)}\big\}$, then
the L{\'e}vy measure $\nu$ of $X$ has finite support of at most
$n-1$ points. 
In other words, there exist an integer  $k\le n-1$ and
constants $a_1,\, \ldots ,\, a_k$ such that 
\begin{eqnarray}\label{eq:jump_diffusion}
X_t & =\alpha t + \sigma W_t +  \sum_{i=1}^{P_t} Y_i-\alpha_1  t, 
\end{eqnarray}
where $P$ is a standard Poisson process, the $Y_i,\, i\in\mathbb N$,  
are independent identically distributed
random variables, independent of $P$, 
with values in $\{a_1,\, \ldots , a_{k}\}$, and
$\alpha_1:=\int_{\mathbb R} x\, \nu(\mathrm{d}x)$.
Conversely, if $X$ has the form given in equation \eqref{eq:jump_diffusion}, 
then $[X]^{(n)}$  is
in the linear span of $\big\{t,\, [X]^{(2)},\,\ldots, \,\  [X]^{(k+1)}\big\}$ 
for all $n\ge k+2$. 
\end{proposition}

\begin{proof}  First assume there are constants 
 $c_0, \, c_1,\, \ldots, c_{n-1}$ such that
$[X]^{(n)}_t= c_0 t +\sum_{k=1}^{n-1}c_k [X]^{(k)}$.
It follows that $\big(\Delta X\big)^{n}= 
\sum_{k=1}^{n-1}c_k \big(\Delta X\big)^{k}$. 
Hence the jumps of $X$ satisfy 
$\big(\Delta X\big)^{n-1}- \sum_{k=1}^{n-1}c_k \big(\Delta X\big)^{k-1}\equiv 0$. 
There are at most
$n-1$ distinct 
real roots to this equation, say $a_1,\, \ldots , a_{k}$ with $k\le n-1$.
Thus the support of the L{\'e}vy measure $\nu$ of $X$ is the set
$\{a_1,\, \ldots , a_{k}\}$. 
The L{\'e}vy decomposition theorem (Protter, Theorem IV.42)
implies the jump component $J$ of $X$ is a compensated compound 
Poisson process 
$J_t=\sum_{i=0}^{P_t} Y_i-\alpha_1  t$, where $Y_i$, 
$i\in\mathbb N$, is a sequence of independent 
random variables with values in $\{a_1,\, \ldots , a_{k}\}$.  

Second, for the converse result, assume that $X$ has the form stated.
For $n\ge 2$ the $n$-power bracket is given by
$[X]^{(n)}_t= \sum_{i=1}^{P_t} Y_i^n+\sigma^21_{\{n=2\}}t$.
Let $c_1, \,\ldots , \, c_{k+1}$ be constants such that  
$a_1,\, \ldots , a_{k}$ are the real roots of the polynomial
$c_1+ c_2 x + \ldots + c_{k+1}x^k$. Then 
$c_{k+1}Y_i^k=- c_1- c_2 Y_i^1-  \ldots - c_{k}Y_i^{k-1} $, and hence 
$c_{k+1}Y_i^{k+2}=- c_1Y_i^2- c_2 Y_i^3-  \ldots - c_{k}Y_i^{k+1}$.
It follows that 
$c_{k+1}[X]^{(k+2)}_t   =\, c_{k+1}\sum_{i=1}^{P_t} Y_i^{k+2} 
 = \, - \sum_{j=2}^{k+1}c_{j-1}\sum_{i=1}^{P_t}Y_i^j
 = \, - \sum_{j=2}^{k+1}c_{j-1} [X]^{(j)}_t + c_1 \sigma^2 t$.
Hence $[X]^{(k+2)}$ 
and consequentially all higher power brackets
$[X]^{(n)}$ for $n\ge k+2$  are linear combinations of 
$t, [X]^{(2)},\, \ldots ,\, [X]^{(k+1)}$.
\end{proof}

\begin{remark}
Since compound Poisson processes that assume finitely many values can be
expressed as linear combinations of independent standard Poisson
processes,  
a L{\'e}vy process $X$ of the form in
equation \eqref{eq:jump_diffusion} is given equivalently  by
$X_t=\alpha t + \sigma W_t + a_1 \bar{P}^1_t + \ldots + a_k \bar{P}^k(t)$,
where $\bar{P}^i$
are independent compensated Poisson processes,  $i=1, \ldots , k$.
\end{remark}

\section{Construction of a quasi-shuffle algebra}\label{quasi_shuffle}
We recall the definition of a quasi-shuffle algebra following
the exposition in Hoffman and Ihara (2012).

Let $\Ab$ be a countable alphabet.
Let $\Rb\Ab$ denote the vector space with $\Ab$ as basis.
We suppose there is commutative associative product
$[\,\cdot\,,\,\cdot\,]$ on $\Rb\Ab$.
Let $\Rb\la\Ab\ra$ denote the algebra of noncommutative polynomials and
formal power series over the field of real numbers $\Rb$ generated by 
monomials (or words) $w=a_1a_2\cdots a_n$ with $a_i\in\Ab$. 

\begin{definition} {\bf (Quasi-shuffle product)}\label{def:quasi_shuffle}
The \emph{quasi-shuffle product} $\ast$
is defined recursively on $\Rb\la\Ab\ra$ via
$va\ast wb=(v*wb)a+(va\ast w)b+(v\ast w)[a,b]$,
where  $v$ and $w$ are words and $a$ and $b$ are letters.
\end{definition}

\begin{definition}
For a minimal family of semimartingales 
${\mathcal X}=\{X^1, \ldots ,\, X^d\}$, 
we define the countable alphabet ${\mathbb A}$ inductively as follows:
\begin{enumerate}
\item[(1)] $\Ab$ {\em contains} the letters $1,\, \ldots ,\, d$.
\item[(2)]  
Inductively, for $n\ge 2$ and for 
$k_1 \le k_2 \le \ldots \le k_n\in\{1, \ldots ,\, d\}$  
consider the 
nested quadratic covariation process
$[X^{k_1}, [X^{k_2}, [ \ldots [X^{k_{n-1}}, X^{k_n}]\ldots ]$. 
If 
this process is not in 
the linear span of 
${\mathcal X}$ and 
previously constructed ones, 
then assign it a new letter.
\end{enumerate}
\end{definition}

\begin{remark}
In general the alphabet  ${\mathbb A}$ is  not
 finite. 
\end{remark}

Let $\mu:\Ab \to {\mathcal A}$ denote the map that identifies a letter with
the corresponding semimartingale in ${\mathcal X}$ or one of the 
power bracket processes identified in (2) above. 
Let $\R\la\Ab\ra$ denote the set of all noncommutative polynomials
and formal series on the alphabet $\Ab$ over $\R$. 
 We extend the map $\mu$ defined above as follows:
for a word  
$w=a_1\ldots a_n\in  \R\la\Ab\ra$ with letters
$a_1,\, \ldots ,\, a_n$ we set $\mu(w)=I_w$ where
$I_w$ is the multiple integral
$I_w(t)\equiv\int_0^t\cdots\int_0^{\tau_{n-1}-}
\mathrm{d}I_{a_1}(\tau_n)\,\cdots\,\mathrm{d}I_{a_n}(\tau_1)$
with 
$I_{a_i}=\mu(a_i)$, $i=1,\ldots , \, n$,
 {\it i.e.}~if $a\in\{1,\ldots ,d\}$ then $I_a=X^a$ or if 
 $a$ is from (2) then 
$I_a$ is a nested quadratic covariation process. 
 We extend this to 
$\R\la\Ab\ra$  linearly. We can now pullback the 
multiplication of multiple integrals to define a product on $\R\la\Ab\ra$ 
as follows.
For words $v$ and $w$ and letters $a$ and $b$, the product of multiple
integrals satisfies 
$I_{va} I_{wb} (t) 
  =  \,  \int_0^t I_v(\tau-) \int_0^{\tau-} \!I_w(s-)\, \mathrm{d} I_b(s)
\, \mathrm{d}I_a(\tau) 
+ \int_0^t \int_0^{\tau-} \! I_v(s-)\, \mathrm{d} I_a(s)\, 
\cdot I_w(\tau-)  \,\mathrm{d}I_b(\tau)
+ \int_0^t  I_v(\tau-) I_w(\tau-)\, \mathrm{d} [I_a,I_b](\tau)$.
Recall that the bracket process defines a commutative, associative product
on the real vector space generated by ${\mathbb A}$ 
(see Remark~\ref{variation_properties}).
Thus the pullback under $\mu$ defines a 
quasi-shuffle  
product $\ast$  
on $\R\la\Ab\ra$ 
where  we set
$[a,b]  := \, \mu^{-1}\big([\mu(a), \mu(b)]\big)$.
We summarize our findings in the following key theorem. 

\begin{theorem}
The map $\mu$ is an algebra 
isomorphism between the quasi-shuffle algebra
$\big(\R\la\Ab\ra,\,\ast)$ and the algebra generated 
by the minimal family $\{X^1,\, \ldots ,\, X^d\}$.\\
\end{theorem}

We now give  $\big(\R\la\Ab\ra,\,\ast)$ an additional structure
via a grading.

\begin{definition} {\bf (Grading)}
On the alphabet ${\mathbb A}$ we define a grading $g$ as follows:
For $a\in\{1, \ldots ,\, d\}$, we only require $g(a) \in \mathbb N$, 
e.g. $g(a)=1$.
For a letter $a \in {\mathbb A}$ with $\mu(a)=[X^i, X^j]$, set
$g(a)=g(i)+g(j)$. More generally,  for a letter $a \in {\mathbb A}$
that is mapped under $\mu$ to a
 nested quadratic covariation process, set $g(a)$ to be the
sum of the gradings of each of its components.  
\end{definition}

\begin{remark}
For each $n\ge 1$ there are only finitely many letters of grade $n$.
Thus ${\mathbb A}$ equipped with the grading $g$ is locally finite. 
 We extend the 
grading $g$ to words $w=a_1\ldots a_n$
by setting 
$g(w)=g(a_1)+\ldots \, + g(a_n)$.
Then 
  $\big(\R\la\Ab\ra,\,\ast)$ equipped 
with the grading $g$
 is a filtered quasi-shuffle algebra, {\it i.e.}~the 
quasi-shuffle product of any two words 
$v$ and $w$
is a linear combination of words with degree $g(v)+g(w)$ or less
(Lang 2002; p.~172). 
 An algebra 
equipped 
with a grading $g$ is a {\em graded algebra}, if
for any two words $v$ and $w$, the product $v\ast w$ is a linear combination
of words with degree $g(v)+g(w)$ (Lang 2002; p.~172). 
Hence 
$(\,\R\la\Ab\ra,\,\ast)$
is a graded algebra if and only if 
for any two letters $a, \, b\in\Ab$ with $[a, b]\not= 0$
the grading $g$ satisfies $g([a,b])= g(a)+g(b)$. 
\end{remark}

\begin{corollary} {\bf (L{\'e}vy Processes)}
Suppose that  $\{X^1, \ldots ,\, X^d\}$ are
 independent (non-deterministic) L{\'e}vy processes that have finite moments.
Then:
\begin{enumerate}
\item[(a)]
The algebra generated by the minimal family 
${\mathcal X}=\{X^1, \ldots ,\, X^d\}$
is isomorphic to the quasi-shuffle algebra  
$\R\la\Ab\ra$, where the alphabet $\Ab$ is defined via (1)
and (2) above. 
If none of the $X^i$ is continuous, then this  holds true
for the minimal family 
${\mathcal X}=\{X^0, X^1, \ldots ,\, X^d\}$
 with $X^0_t:= t$. 
\item[(b)] The alphabet $\Ab$ in (a) is finite, 
if and only if for each $i=1,\ldots ,\, d$,
the purely discontinuous martingale part  of 
$X^{(i)}$ is either identically zero or 
a linear combination of independent standard compensated
Poisson processes. 
\item[(c)] Suppose that the grading $g$  
is specified  on $\{0, 1,\ldots ,\, d\}$ as
$g(0) =  2$ and $g(i) = 1$  for  $i=1,\ldots ,\,d$.
Then the 
algebra $\big(\R\la\Ab\ra, \ast\big)$ 
equipped with the grading $g$ 
is a graded algebra unless 
the  purely discontinuous martingale part of one or more
of the $X^{(i)}$, $i=1, \ldots, d$, 
is a linear combination of independent standard compensated Poisson processes. 
\end{enumerate}
\end{corollary}

\begin{proof}
Result (a) follows from Theorem~\ref{th:main_Levy}  and  
result (b) from Proposition~\ref{compound_Poisson}.  
For result (c), we have to show that 
for any letters $a, \, b\in\Ab$ with $[a, b]\not= 0$
the grading $g$ satisfies $g([a,b])= g(a)+g(b)$. Since
$[X^i, X^j]\equiv 0$ for $i\not= j$, we can assume that
$a$ and $b$ have  corresponding semimartingales
$\mu(a)=[X^i]^{(n)}$ and $\mu(b)=[X^i]^{(m)}$. 
We have by definition 
$\mu([a, b]) =  \big[[X^i]^{(n)}, [X^i]^{(m)}\big]
= \, [X^i]^{(n+m)}$.
If $X^i$ is continuous, say $X^i_t=\alpha t +\sigma W_t$, then
$[X^i]^{(n+m)}=\sigma^21_{\{n=m=1\}} \cdot t$. Hence if $\sigma\not= 0$ then
$g([a,b])=2$ for $n=m=1$, and $[a,b]$ is zero otherwise. Hence result
(c) 
follows for continuous $X^i$.
Suppose now that $X^i$ is not continuous.
By definition of $g$ we have
$g([a, b])= (n+m)g(i)=g(a)+g(b)$, 
if  $[X^i]^{(n+m)}$ is not in the linear span generated by $t$ and 
$[X^i]^{(k)}$ with $k\le n+m-1$.
On the other hand if  $[X^i]^{(n+m)}$ is in the linear span
generated by $t$ and 
$[X^i]^{(k)}$ with $k\le n+m-1$,  then $[a, b]$ is a linear combination
 of letters with degree of at most  $(n+m-1) g(i)$. 
By Proposition~\ref{compound_Poisson} the power bracket
$[X^i]^{(n+m)}$ is in the linear span generated by $t$ and lower order power
brackets if and only if 
 the purely discontinuous martingale part of one or more
of the $X^{(i)}$, $i=1, \ldots, d$, 
is a linear combination of independent standard compensated Poisson processes.
Hence assertion (c) follows.
\end{proof}

\section{Conclusions and further work}\label{conclusion}
The main results we have proved, that minimal families of 
semimartingales form a quasi-shuffle algebra and 
a family of independent L\'evy processes generate
such a minimal family, are important in their own right. 
However, there are important further implications and applications 
we intend to pursue, see Curry \emph{et al.\/} (2013). 
First the Hoffman exponential map
gives an isomorphism between the shuffle and quasi-shuffle algebras.
This simplifies the algebra and analysis and raises a natural question. 
Would the corresponding shuffle algebra be based on the Marcus integral 
(see Marcus 1981 or Applebaum 2009)? 
Second, with deconcatenation as a coproduct and a natural
antipode established therefrom, the quasi-shuffle algebra becomes 
a \emph{Hopf algebra}. 
Thus in principle we can establish the quasi-shuffle convolution
algebra of endomorphisms on the quasi-shuffle Hopf algebra. 
See Reutenauer (1993, p.~58) for the shuffle
case and Novelli, Patras and Thibon (2011) for the quasi-shuffle case.
The convolution algebra is a natural setting for designing
numerical methods for stochastic differential equations. See 
Ebrahimi--Fard \emph{et al.\/ } (2012) where efficient numerical
methods for stochastic differential equations driven by Wiener 
processes are constructed utilizing the convolution shuffle algebra. 
Hence our next goal is to construct efficient numerical methods
for stochastic differential equations driven by L\'evy processes. 


\end{document}